\documentclass[a4paper,12pt]{article}
\usepackage[utf8]{inputenc}
\usepackage{amsmath}
\usepackage{latexsym}
\usepackage{amsfonts}
\usepackage{amssymb}
\usepackage{amscd}
\usepackage{theorem}
\usepackage{a4wide}
\usepackage{color}

\newtheorem{theorem}{Theorem}[section]
\newtheorem{lemma}[theorem]{Lemma}
\newtheorem{proposition}[theorem]{Proposition}
\newtheorem{corollary}[theorem]{Corollary}
\newtheorem{definition}[theorem]{Definition}
  \newtheorem{example}[theorem]{Example}

  \newtheorem{remark}[theorem]{Remark}

\newenvironment{proof}{    
  \noindent
  \textbf{Proof.}}{
  \hfill $\Box$
  \vspace{3mm}
}

\numberwithin{equation}{section}

\newcommand{\norm}[1]{{\left\|#1\right\|}}

\newcommand{\N}{\mathbb{N}} 
\newcommand{\C}{\mathbb{C}} 
\newcommand{\K}{\mathbb{K}} 
\newcommand{\D}{\mathbb{D}} 

\DeclareMathOperator*{\ind}{ind}    

\title{Frames and representing systems in Fr\'echet spaces and their duals}
\author{J. Bonet, C. Fern\'andez, A. Galbis, J.M. Ribera}
\date{}
\begin{document}

\maketitle

\vspace{1cm}
\begin{abstract}
Frames and Bessel sequences in Fr\'echet spaces and their duals are defined and studied. Their relation with Schauder frames and representing systems is analyzed. The abstract results presented here, when applied to concrete spaces of analytic functions, give many examples and consequences about sampling sets and Dirichlet series expansions.
\end{abstract}

\renewcommand{\thefootnote}{}
\section{Introduction and preliminaries}

The purpose of this paper is twofold. On the one hand we study $\Lambda$-Bessel sequences $(g_i)_i \subset E'$, $\Lambda$-frames and frames with respect to $\Lambda$ in the dual of a Hausdorff locally convex space $E$, in particular for Fr\'echet spaces and complete $\left(LB\right)$-spaces $E$, with $\Lambda$ a sequence space. We investigate the relation of these concepts with representing systems in the sense of Kadets and Korobeinik \cite{Kadets-Korobeinik-Studia} and with the Schauder frames, that were investigated by the authors in \cite{bfgr}. On the other hand our article emphasizes the deep connection of frames for Fr\'echet and $\left(LB\right)$-spaces with the sufficient and weakly sufficient sets for weighted Fr\'echet and $\left(LB\right)$-spaces of holomorphic functions. These concepts correspond to sampling sets in the case of Banach spaces of holomorphic functions. Our general results in Sections \ref{sec-general} and  \ref{sec:LB} permit us to obtain as a consequence many examples and results in the literature in a unified way in Section \ref{Examples}.

Section \ref{sec-general} of our article is inspired by the work of Casazza, Christensen and Stoeva \cite{casazza_jmaa} in the context of Banach spaces. Their characterizations of Banach frames and frames with respect to a $BK$-sequence space gave us the proper hint to present here the right definitions in our more general setting; see Definition \ref{FR_Pre_Def_Fr}. The main result of this section is Proposition \ref{frame-repsystem}. A point of view different from ours concerning frames in Fr\'echet spaces was presented by Pilipovic and Stoeva \cite{pilipovic_stoeva_2011} and \cite{pilipovic_stoeva_2014}. Banach frames were introduced by Gr\"ochenig.
Shrinking and boundedly
complete Schauder frames
for Banach spaces were studied by Carando, Lasalle and Schmidberg
\cite{carando_lassalle}. Other precise references
to work in this direction in the Banach space setting can be seen in
\cite{casazza_jmaa}. Motivated by the applications to weakly sufficient sets for weighted $\left(LB\right)$-spaces of holomorphic functions we present several abstract results about $\Lambda$-frames in complete $\left(LB\right)$-spaces, that require a delicate analysis, in Section \ref{sec:LB}. Our main result is Theorem \ref{frames_DFS}. Finally, many applications, results and examples are collected in Section \ref{Examples} concerning sufficient sets for weighted Fr\'echet spaces of holomorphic functions and weakly sufficient sets for weighted $\left(LB\right)$-spaces of holomorphic functions. We include here consequences related to the work of many authors; see \cite{abanin_khoi}, \cite{Abanin_Varziev}, \cite{Bonet_Domanski_2003_Sampling}, \cite{khoi-thomas}, \cite{korobeinik2}, \cite{melikhov}, \cite{napalkov1}, \cite{Schneider} and \cite{Taylor_1972_Discrete}.

Throughout this work, $E$ denotes a locally convex Hausdorff linear topological space (briefly, a lcs) and $cs(E)$ is the system of continuous seminorms describing the topology of $E.$ Sometimes additional assumptions on $E$ are added. The symbol $E'$ stands for the topological dual of $E$ and $\sigma(E',E)$ for the weak* topology on $E'$. We set  $E'_\beta$ for the dual $E'$ endowed with the topology $\beta(E',E)$ of uniform convergence on the bounded sets of $E.$ We will refer to $E'_\beta$ as the strong dual of $E.$ The Mackey topology $\mu(E',E)$ is the topology on $E'$ of the uniform convergence on the absolutely convex and weakly compact sets of $E.$ Basic references for lcs are \cite{jarchow} and \cite{Meise_Vogt_1997_Introduction}. If $T:E \rightarrow F$ is a continuous linear operator, its transpose is denoted by $T':F' \rightarrow E'$, and it is defined by $T'(v)(x):=v(T(x)), x \in E, v \in F'$.  We recall that a Fr\'echet space is a complete metrizable lcs. An $\left(LB\right)$-space is a lcs that can be
represented as an injective inductive limit of a sequence $\left(E_n\right)_n $ of Banach spaces. In most of the results we need the assumption that the lcs is barrelled. The reason is that Banach-Steinhaus theorem holds for barrelled lcs. Every Fr\'echet space and every $\left(LB\right)$-space is barrelled. We refer the reader to  \cite{jarchow} and \cite{PCB} for more information about barrelled spaces.
\par\medskip\noindent
As usual $\omega$ denotes the countable product $\K^{\N}$ of copies of the scalar field, endowed by the product topology, and  $\varphi$ stands for the space of sequences with finite support. A sequence space $\bigwedge$ is a lcs which contains $\varphi$ and is continuously included in $\omega.$

\begin{definition}{\rm Given a sequence space $\Lambda$ its $\beta$-dual space is defined as
$$\Lambda^{\beta} := \left\{(y_i)_i \in \omega : \sum_{ i = 1}^{\infty} x_i y_i\ \mbox{converges for every}\ (x_i)_i \in \Lambda \right\}.$$
}
\end{definition}

Clearly, $(\Lambda, \Lambda ^\beta)$ is a dual pair. Under additional assumptions we even have the  relation given in the next essentially known lemma.

\begin{lemma}\label{FR_Pre_Lem_DualS}
Let $\Lambda$ be a barrelled sequence lcs for which the canonical unit vectors $( e_i)_i$ form a Schauder basis. Then, its topological dual $\Lambda'$ can be algebraically identified with its $\beta$-dual $\Lambda^{\beta}$ and the canonical unit vectors $( e_i)_i$ are a basis for $(\Lambda^\beta, \mu(\Lambda^\beta, \Lambda)).$
Moreover if we consider  on $\Lambda^{\beta} $  the system of seminorms given by $$p_B( (y_i)_i):=\sup_{x \in B} \left| \sum_i x_iy_i\right| ,$$ where $B$ runs on the bounded subsets of $\Lambda$ then $\left( \Lambda^\beta,\left(p_B\right)_B\right)$ is topologically isomorphic to $  \left( \Lambda', \beta\left(\Lambda', \Lambda\right)\right)$.

\end{lemma}

\begin{proof}
The map $$\psi : \Lambda' \to \Lambda^{\beta}, \, h\to (h(e_i))_i$$ is a linear bijection. In fact, for every $h\in\Lambda^\prime$ and $x=(x_i)_i \in \Lambda$ we have $h(x)=\sum_ix_ih(e_i)$ which implies that $\psi$ is well defined and obviously linear and injective. The barrelledness of $\Lambda$ and the Banach-Steinhauss theorem give the surjectivity.

If $K\subset \Lambda$ is $\sigma(\Lambda, \Lambda^\beta)$-compact, given $y\in \Lambda^\beta$ and $\varepsilon >0$ there is $n_0$ such that for $n \geq n_0,$ $$\left|\sum_{i=n}^\infty y_ix_i\right|<\varepsilon$$ for all $x\in K,$ from where $y=\sum_iy_ie_i$ in the Mackey topology $ \mu(\Lambda^\beta, \Lambda).$

As for each bounded subset $B$ of $\Lambda$ we have $$p_B((h(e_i))_i)=\sup_{x\in B}|h(x)|,$$ the topological identity follows.
\end{proof}
\par\medskip
From now on, if the sequence space $\Lambda$ satisfies the assumption in Lemma \ref{FR_Pre_Lem_DualS}, we identify $\Lambda '$ with $\Lambda^\beta$ and use always $\Lambda '.$

\section{General results}\label{sec-general}

\begin{definition}\label{FR_Pre_Def_Fr}{\rm
Let $E$ be a lcs and $\Lambda$ be a sequence space.
\begin{enumerate}

\item  $(g_i)_i \subset E'$ is called a \textit{$\Lambda$-Bessel sequence} in $E'$ if the analysis operator
$$
    \left.
	\begin{array}{rll}
		  U_{(g_i)_i} : E  &\longrightarrow  & \Lambda \\
		  x &\longrightarrow &  (g_i\left(x\right))_i
	\end{array}
    \right.
    $$ is continuous.

\item $(g_i)_i \subset E'$ is called a \textit{$\Lambda$-frame} if the analysis operator $U=U_{(g_i)_i}$ is an isomorphism into. If in addition  the range $U(E)$ of the analysis operator is complemented in $\Lambda$, then $(g_i)_i$ is said to be a \textit{frame for $E$ with respect to $\Lambda.$} In this case there exists $S: \Lambda \rightarrow E$ such that $S \circ U = \left. id \right|_E$.

\end{enumerate}
}
\end{definition}

For simplicity, when $(g_i)_i$ is clear from the context, the analysis operator will be denoted by $U$.

Definition \ref{FR_Pre_Def_Fr} is motivated by Definitions 1.2 and 1.3 in \cite{casazza_jmaa}. More precisely we had in mind Theorems 2.1 and 2.4 in \cite{casazza_jmaa} that give the right idea how to extend the definitions to the locally convex setting.

Clearly, given a lcs $E$ each sequence $( g_i )_i \subset E'$ is an $\omega$-Bessel sequence. On the other hand, if $\Lambda$ is a Hilbert space and $E$ is complete, each $\Lambda$-frame is a frame with respect to $\Lambda.$ Obviously a lcs space has a $\Lambda$-frame if and only if it is isomorphic to a subspace of $\Lambda$ and it has  a frame with respect to $\Lambda$ if and only if it is isomorphic to a complemented subspace of $\Lambda.$ Therefore, the property of having a $\Lambda$-frame is inherited by subspaces whereas having a frame with respect to $\Lambda$ is inherited by complemented subspaces.

\begin{remark}{\rm Let $E$ be a lcs, $\Lambda_1,$ $\Lambda_2$ sequence spaces, $( g_i )_i \subset E'$ a  $\Lambda_1$-Bessel  sequence and $( h_i )_i \subset E'$ a $\Lambda_2$-frame. We define $(f_k)_k \subset E'$ as $f_k=g_i$ for $k=2i-1$ and $f_k=h_i$ when $k=2i.$ Consider the sequence space $$\Lambda:=\{ (\alpha_k)_k: \, (\alpha_{2k-1})_k\in \Lambda_1, \mbox{ and } (\alpha_{2k})_k\in \Lambda_2\},$$  with the topology given by the seminorms
$$||\alpha||_{p,q}:=p((\alpha_{2k-1})_k)+q((\alpha_{2k})_k), \mbox{ where } p\in cs(\Lambda_1) , \, q\in cs(\Lambda_2).$$ Then $(f_k)_k$ is a $\Lambda$-frame for $E.$ In the case that $\Lambda_1 = \Lambda_2$ is one of the spaces $c_0$ or $\ell_p$ then $\Lambda = \Lambda_1 = \Lambda_2.$
}
\end{remark}

Let $E$ be a lcs, $(x_i)_i\subset E$ and $(x_i')_i\subset E' $. We recall that $\left( ( x_i' )_i , ( x_i )_i \right)$ is said to be a \textit{Schauder frame} of $E$  if
\begin{equation*}\label{seminorms}
 x = \sum_{i = 1}^{\infty} x_i' \left( x \right) x_i , \quad \mbox{ for all } x \in E,
\end{equation*}
the series converging in $E$ \cite{bfgr}. The associated sequence space is
$$
\Lambda := \left\{\alpha = ( \alpha_i )_i \in \omega : \sum_{ i = 1}^{\infty} \alpha_i x_i\ \mbox{is convergent in}\  E \right\}.$$ Endowed with the  system of seminorms
        \begin{equation}\label{eq_AD_Lem_DefiSeqSpace_Semi}
            {\mathcal Q} := \left\{ q_p\left(\left(\alpha_i\right)_i\right) := \sup_n p\left(\sum_{i = 1}^n \alpha_i x_i \right) , \mbox{ for all } p \in cs(E) \right\},
        \end{equation} $\Lambda$ is a sequence space and the canonical unit vectors form a Schauder basis (\cite[Lemma 1.3]{bfgr}).
\par\medskip

There is a close connection between $\Lambda$-frames and Schauder frames.

\begin{proposition} \begin{itemize}
                         \item[(a)] Let $\left( \left( x_i'\right)_i, \left(x_i\right)_i \right)$ be a Schauder frame for a barrelled and complete lcs $E$ and $\Lambda$ the associated sequence space. Then $( x_i')_i \subset E'$ is a frame for $E$ with respect to $\Lambda.$ If moreover $\Lambda$ is barrelled then $( x_i)_i \subset E$ is a frame for $E^\prime$ with respect to $\Lambda^\prime.$
\item[(b)] If $( x_i')_i \subset E'$ is a frame for $E$ with respect to a sequence space $\Lambda ,$ and $\Lambda$ has a Schauder frame, then $E$ also admits a Schauder frame.
                        \end{itemize}
 \end{proposition}

\begin{proof} (a) According to the proof of \cite[Theorem 1.4]{bfgr} the operators $U:E \to \Lambda$ and $S:\Lambda \to E$ given by $U(x):= (x_i'(x))_i$ and $S((\alpha_i)_i) := \sum_{i = 1}^{\infty} \alpha_i x_i$ respectively, are continuous and $S \circ U = id_E$. Consequently $( x_i')_i$ is a frame for $E$ with respect to $\Lambda.$ Under the additional assumption that $\Lambda$ is barrelled we have that $\Lambda^\prime = \Lambda^\beta$ is a sequence space. Moreover $S'(x'):= (x'(x_i))_i$  for each $x'\in E^\prime$ and from $U' \circ S' = id_{E'}$ we conclude that $( x_i)_i$ is a frame for $E'$ with respect to $\Lambda'.$
\par
Statement (b) follows from the fact that having a Schauder frame is inherited by complemented subspaces.
\end{proof}

\par\medskip
The barrelledness of the sequence space $\Lambda$ naturally associated to a Schauder frame follows for instance if $E$ is a Fr\'echet space. Observe that the dual space $E'$ need not be separable, in which case neither need $\Lambda^\prime$ be.

\begin{definition}{\rm (\cite{Kadets-Korobeinik-Studia}) A \textit{representing system} in a lcs  $E$ is a sequence $(x_i)_i$ in $E$ such that each $x\in E$ admits a representation
$$x=\sum_i c_i x_i$$ the series converging in $E.$
}
\end{definition}

The coefficients in the representation need not be unique, that is, one  can have $$0=\sum_i d_i x_i$$ for a non-zero sequence $(d_i)_i.$  Moreover, we do not assume that it is possible to find  a representation of this type with coefficients depending continuously on the vectors.

Clearly each topological basis is a representing system. Given a  Schauder frame $\left( \left( x_i'\right)_i, \left(x_i\right)_i \right)$, the sequence $\left(x_i\right)_i$ is a representing system. However, there are representing systems that are neither basis nor coming from a Schauder frames. In fact,   each separable Fr\'echet space  has a representing system \cite[Theorem 1]{Kadets-Korobeinik-Studia} but only those separable Fr\'echet spaces with the bounded approximation property admit a Schauder frame \cite[Corollary 1.5]{bfgr}.

\begin{definition}{\rm A $\Lambda$-{\it representing system} in a lcs  $E$ is a sequence $(x_i)_i$ in $E$ such that each $x\in E$ admits a representation
$x=\sum_i c_i x_i$ with $(c_i)_i\in \Lambda.$
}
\end{definition}

\begin{proposition}\label{frames-representing} Let $E$ be a barrelled lcs and let $\Lambda$ be a barrelled  sequence lcs for which the canonical unit vectors $( e_i)_i$ form a Schauder basis. Then
\begin{itemize}
\item[(i1)]
 $( g_i )_i \subset E'$ is a $\Lambda$-Bessel sequence if and only if the operator
 $$
 T: (\Lambda', \mu(\Lambda',\Lambda)) \to (E', \mu(E', E)),\ (d_i)_i \mapsto \sum_{i = 1}^{\infty} d_i g_i$$ is well defined and continuous.
 \item[(i2)] $( g_i )_i \subset E'$ is $\Lambda'$-Bessel for $E$ if and only if the operator $T:\Lambda \to (E', \beta(E',E))$ given by $T((d_i)_i) := \sum_{i = 1}^{\infty} d_i g_i$ is well defined and continuous.
 \item[(ii)] If $( g_i )_i\subset E'$ is a $\Lambda$-frame in $E,$ then $( g_i )_i$ is a $\Lambda'$-representing system for $(E', \mu(E',E)).$ If moreover $E$ is reflexive, then $( g_i )_i$ is a $\Lambda'$-representing system for $(E', \beta(E',E)).$
 \item[(iii)] If $( g_i )_i\subset E'$ is a $\Lambda$-Bessel sequence which is also a $\Lambda'$-representing system for $(E', \mu(E',E))$ then $( g_i )_i$ is a $(\Lambda, \sigma(\Lambda, \Lambda'))$-frame for $(E,\sigma(E,E')).$
\item[(iv)] If in addition $E$ and $\Lambda$ are Fr\'echet spaces, then $( g_i )_i\subset E'$ is a $\Lambda$-frame for $E$ if, and only if, $( g_i )_i$ is $\Lambda$-Bessel and a $\Lambda'$-representing system for $(E', \mu(E',E)).$
\end{itemize}

\end{proposition}

\begin{proof} (i1) Let us assume that $( g_i )_i$ is a $\Lambda$-Bessel sequence and consider $T=U'$ the transposed map of the analysis operator $U:E \to \Lambda,$ $U(x) = (g_i(x))_i.$ Then $T:(\Lambda', \mu(\Lambda',\Lambda)) \to (E', \mu(E', E))$ is continuous and $T(e_i)=g_i.$ As the canonical unit vectors are a basis for  $(\Lambda', \mu(\Lambda',\Lambda))$ we conclude $T((d_i)_i) = \sum_{i = 1}^{\infty} d_i g_i.$ Conversely, if $T$ is a well defined and continuous map, then its transposed $T':E\to \Lambda$ is also continuous which means that $( g_i )_i$ is a $\Lambda$-Bessel sequence. (i2) is proved similarly considering that the dual $(\Lambda', \beta(\Lambda', \Lambda))$ is a sequence space.
\par
(ii) If $( g_i )_i$ is a $\Lambda$-frame then $U$ is a topological isomorphism into, hence $T = U'$ is surjective. In particular $( g_i )_i$ is a $\Lambda'$-representing system in $(E', \mu(E', E)).$
\par
(iii) From (i), the map $T: (\Lambda', \mu(\Lambda',\Lambda)) \to (E', \mu(E', E)),\ (d_i)_i \mapsto \sum_{i = 1}^{\infty} d_i g_i,$ is well defined, continuous and surjective. Consequently $T':(E,\sigma(E,E'))\to (\Lambda, \sigma(\Lambda, \Lambda'))$ is an isomorphism into \cite[9.6.1]{jarchow}, hence $( g_i )_i$ is a $(\Lambda, \sigma(\Lambda, \Lambda'))$-frame for $(E,\sigma(E,E')).$
\par
(iv) Necessity follows from (ii) and sufficiency follows from the closed range theorem \cite[9.6.3]{jarchow} and (iii).
\end{proof}

Relevant consequences of Proposition \ref{frames-representing} for spaces of analytic functions are given later in Theorem \ref{abanin_khoi}, that is due to Abanin and Khoi \cite{abanin_khoi}, and Corollary \ref{cor:rep-Ap}.

The next result is the extension in our context of \cite[Corollary 3.3]{casazza_jmaa}.

\begin{proposition} Let $E$ be a reflexive space and let $\Lambda$ be a reflexive sequence space for which the canonical unit vectors $( e_i)_i$ form a Schauder basis. If either
\begin{itemize}
 \item[(i)] $E$ and $\Lambda$ are Fr\'echet spaces \par or
\item[(ii)] $E$ is the strong dual of a Fr\'echet Montel space and $\Lambda$ is an $\left(LB\right)$-space,
\end{itemize} then $( g_i )_i\subset E'$ is a $\Lambda$-frame for $E$ if, and only if, $( g_i )_i$ is $\Lambda$-Bessel and a $\Lambda'$-representing system for $(E', \beta(E',E)).$
\end{proposition}
\begin{proof} The case (i) is Proposition \ref{frames-representing} (iv). Only the sufficiency in case (ii) has to be proved. Let us assume that $( g_i )_i$ is $\Lambda$-Bessel and a $\Lambda'$-representing system for $(E', \beta(E',E))$ and consider the continuous map $U:E\to \Lambda,$ $U(x) = (g_i(x))_i.$ Then $T = U':\Lambda'\to E'$ is a well-defined continuous and surjective map. Since $E'$ is a Fr\'echet Montel space, the map $T$ lifts bounded sets, that is, for every bounded set $B$ in $E'$ we can find a bounded set $C$ in $\Lambda'$ such that $B\subset T(C).$ Hence $U:E\to \Lambda$ is a topological isomorphism into, which means that $( g_i )_i$ is a $\Lambda$-frame for $E.$
\end{proof}

\begin{example}\label{ex:repsystemFM}{\rm Let $F$ be the strong dual of a Fr\'echet Montel space $E.$ Since $E$ is separable it admits a representing system $(g_i)_i\subset F'.$ Consider the Fr\'echet sequence space
$$
\Lambda = \left\{(\alpha_i)_i\in \omega: \, \sum_i\alpha_i g_i\ \mbox{is convergent in}\ E\right\}
$$ endowed with the system of seminorms as in (\ref{eq_AD_Lem_DefiSeqSpace_Semi}). Then $(g_i)_i$ is a $\Lambda$-representing system for $E$ and also a $(\Lambda', \beta(\Lambda',\Lambda))$-frame for $F.$
}
\end{example}
In fact, the continuous map $T:\Lambda \to E,$ $(\alpha_i)_i\mapsto \sum_i\alpha_i g_i,$ lifts bounded sets, which implies that $U = T':E\to \Lambda'$ is an isomorphism into. $\Box$

\begin{example}\label{ex:repsystemDFS}{\rm Let $E$ be a Fr\'echet-Schwartz space. Then there are a  Fr\'echet sequence space $\Lambda$ and a sequence $(g_j)_j\subset E'$ which is a $\Lambda$-frame for $E.$
 }
\end{example}
In fact, $F:= E'_\beta = {\rm ind}_k F_k$ is a sequentially retractive $\left(LB\right)$-space, hence it is sequentially separable and admits a representing system $(g_j)_j$ (\cite[Theorem 1]{Kadets-Korobeinik-Studia}). Now, for each $k$ we put
$$
\Gamma_k:=\left\{ \alpha \in \omega \, : \alpha_jg_j\in F_k\ \mbox{for all}\ j\ \mbox{and}\ \sum_{j=1}^{\infty}\alpha_j g_j \mbox{ converges in } F_k \right\}.$$

Without loss of generality we may assume that $\Gamma_k$ is non-trivial for each $k$ and we endow it with the norm $$ q_k(\alpha)=\mbox{ sup}_n||\sum_{j=1}^{n}\alpha_j g_j ||_k,$$ where $\norm{\cdot}_k$ denotes the norm of the Banach space $F_k.$ Then $(\Gamma_k, q_k),$ $k\in {\mathbb N},$ is an increasing sequence of Banach spaces with continuous inclusions and $\bigcup_k \Gamma_k$ coincides algebraically with $$\Gamma:=\left\{\alpha \in \omega \, : \, \sum_{j=1}^{\infty}\alpha_j g_j \mbox{ converges in } F \right\}.$$ Hence $\Gamma,$ endowed with its natural $\left(LB\right)$-topology, is a sequence space with the property that the canonical unit vectors are a basis. Moreover, the map $$\Gamma \to F, \, \alpha \to \sum_{j=1}^{\infty}\alpha_j g_j $$ is well defined, continuous and surjective. Therefore, $\Lambda = \Gamma'$ is a Fr\'echet sequence space and $(g_j)_j$ is a $\Lambda$-frame for $E.$ $\Box$
\par\medskip
The following result is the version in the locally convex context of Propositions 2.2 and 3.4 of \cite{casazza_jmaa}. It relates $\Lambda$-Bessel sequences with frames with respect to $\Lambda$ when $\Lambda$ is a barrelled sequence space for which the canonical unit vectors $( e_i)_i$ form a Schauder basis. Note that, if $( g_i )_i \subset E'$ is a frame with respect to $\Lambda$, by definition, $U(E)$ is complemented in $\Lambda$; this means that the operator $U^{-1}: U(E) \to E$ can be extended to a continuous linear operator $S:\Lambda \to E$.

\begin{proposition}\label{frame-repsystem}
Let $E$ be a barrelled and complete lcs and let $\Lambda$ be a barrelled sequence space
for which the canonical unit vectors $( e_i)_i$ form a Schauder basis. If $( g_i )_i \subset E'$ is $\Lambda$-Bessel for $E$ then the following conditions are equivalent:
\begin{enumerate}
\item[(i)] $( g_i )_i \subset E'$ is a frame with respect to $\Lambda$.
\item[(ii)] There exists a sequence $( f_i )_i \subset E$, such that $\sum_{ i = 1}^{\infty} c_i f_i$ is convergent for every $(c_i )_i \in \Lambda$ and $x = \sum_{ i = 1}^{\infty} g_i\left(x\right)f_i$ , for every $x \in E$.
\item[(iii)] There exists a $\Lambda'$-Bessel sequence $( f_i )_i \subset E \subseteq E'' $ for $E'$ such that $x = \sum_{i = 1}^{\infty} g_i\left(x\right)f_i$ for every $x \in E$.
\end{enumerate}
If the canonical unit vectors form a basis for both $\Lambda$ and $\Lambda_{\beta}'$, (i)-(iii) are also equivalent to
\begin{enumerate}
\item[(iv)] There exists a $\Lambda'$-Bessel sequence $( f_i )_i \subset E \subseteq E'' $ for $E'$ such that $x' = \sum_{i = 1}^{\infty} x'\left(f_i\right) g_i$ for every $x' \in E'$ with  convergence in the strong topology.
\end{enumerate}
If each of the cases (iii) and (iv) hold then $( f_i )_i$ is actually a frame for $E'$ with respect to $\Lambda'$. Moreover, $((g_i)_i,(f_i)_i)$ is a shrinking Schauder frame.
\end{proposition}
\begin{proof}
We consider $U$ as in Definition \ref{FR_Pre_Def_Fr} which is a continuous map.
\begin{description}
\item[(i) $\to$ (ii)] Let $S:\Lambda \to E$ be a continuous linear extension of $U^{-1}$ such that $S \circ U = \left. I \right|_{E}$. Define $f_i := S\left(e_i\right)$ and observe that, for all $( c_i )_i \in \Lambda,$
$$ \sum_{ i = 1}^{\infty} c_i f_i = \sum_{i = 1}^{\infty} c_i S\left(e_i\right) =  S\left( \sum_{i = 1}^{\infty} c_i e_i \right) = S\left((c_i)_i\right).$$
Moreover, for every $x \in E$,  $x = (S\circ U)\left(x\right) = \sum_{ i = 1}^{\infty} g_i\left(x\right)f_i $.

\item[ (ii) $\to$ (i)] Assume that (ii) is satisfied, we define $S: \Lambda \to E$ by $S\left(( c_i )_i \right) := \sum_{i = 1}^{\infty} c_i f_i $ with $( c_i )_i$. Observe that, by Banach-Steinhaus theorem, $S$ is a continuous operator. Taking $( g_i \left(x\right) )_i \in U(E)$ we obtain
$$ S\left(( g_i\left(x\right))_i\right) = \sum_{ i = 1}^{\infty} g_i\left(x\right) f_i = x.$$
We obtain that $U$ is an isomorphism into and $S$ is a continuous extension of $U^{-1}.$ Hence (i) holds.

\item[ (ii) $\to$ (iii)] Let $V: \Lambda \to E$ be a linear continuous extension of $U_{( g_i )_i}^{-1}$. Set $f_i := V\left(e_i\right)$. By Lemma \ref{FR_Pre_Lem_DualS}, for every $ x' \in E'$ we have $( x'\left(f_i\right))_i = ( x'\left(V\left(e_i\right)\right))_i \in \Lambda'$ and $\left( f_i \right)_i$, considered as a sequence in $E''$, is an $\Lambda'$-Bessel sequence for $E'$. Note that we can also prove the result using that $S': E_{\beta}' \to \Lambda'$, given by $S'(x') := (x'(f_i))_i$ is continuous since it is the transpose of $S$.

\item[(iii) $\to$ (ii)] By Proposition \ref{frames-representing} (i2), if (iii) is valid then the operator $T: \Lambda' \to E \subset E''$ given by $T((c_i)_i) := \sum_{i = 1}^{\infty} c_i f_i $ is well defined and continuous, hence (ii) is satisfied.

Now assume that the canonical unit vectors form a basis for both $\Lambda$ and $\Lambda_{\beta}'.$ Denote the canonical basis of $\Lambda$ by $(e_i)_i$ and the canonical basis of $\Lambda_{\beta}'$ by $(z_i)_i$.

\item[(iii) $\to$ (iv)]  If (iii) is valid, there exists $( f_i)_i \subset E \subseteq E''$ that is $\Lambda'$-Bessel for $E'$ such that $x = \sum_{ i = 1}^{\infty} g_i\left(x\right) f_i$. Observe that, as $( x'\left( f_i \right) )_i$ belongs to $\Lambda'$, then $( x'\left( f_i \right) )_i = \sum_{i = 1}^{\infty} x'\left(f_i\right) z_i$ in $(\Lambda', \beta(\Lambda',\Lambda))$. Given
a bounded set $B \in E$ then $C = \{ (g_i\left(x\right)): x \in B \}$ is a bounded set in $\Lambda$.
If $p_B\in cs(E_{\beta}')$ is the continuous seminorm defined by $p_B (u') := \sup_{x \in B} |u'(x)|$ then

\begin{eqnarray*}
p_B\left( x' - \sum_{ i = 1}^n x'\left(f_i\right)g_i\right) & = & \sup_{ x \in B} \left| x'\left(x\right) - \sum_{ i = 1}^n x'\left(f_i\right)g_i\left(x\right) \right| = \nonumber \\
& = & \sup_{ x \in B} \left| x'\left( \sum_{i = 1}^{\infty} g_i\left(x\right)f_i\right)- \sum_{ i = 1}^n x'\left(f_i\right)g_i\left(x\right) \right| = \nonumber \\
& = & \sup_{ x \in B} \left| \sum_{ i = n + 1}^{\infty} x'\left(f_i\right)g_i\left(x\right)\right| = \sup_{\phi \in C} \left| \phi\left(\sum_{ i = n + 1}^{\infty} x'\left(f_i\right)z_i\right)\right| \nonumber \\
& = & q_C\left( \sum_{ i = n+1}^{\infty} x'\left(f_i\right)z_i \right)
\end{eqnarray*}
where $q_C\in cs(\Lambda')$ is given by $q_C\left(\alpha\right) := \sup_{\phi\in C}\left|\phi(\alpha)\right|$ for every $\alpha \in \Lambda_{\beta}'.$ Then, $q_C\left( \sum_{ i = n+1}^{\infty} x'\left(f_i\right)z_i \right)$ converges to 0 as $n$ converges to infinity since $( x'(f_i) )_i = \sum_{n = 1}^{\infty} x'(f_i) z_i$ in $\Lambda_{\beta}'.$

\item[(iv) $\to$ (iii)] If (iv) is valid, then there exists $( f_i )_i $ a $\Lambda'$-Bessel sequence for $E'$ such that $x' = \sum_{ i = 1}^{\infty} x'\left(f_i\right) g_i$.
Given a bounded subset $B'\subset E'$ then $C' = \{ (x'\left(f_i\right)) : x' \in B' \}$ is a bounded set in $\Lambda_{\beta}'$.
If $p_{B'}\in cs(E)$ is the continuous seminorm defined by $p_{B'} (x) := \sup_{x' \in B'} |x'(x)|$ then
\begin{eqnarray*}
p_{B'}\left( x - \sum_{ i = 1}^n g_i\left(x\right) f_i \right) & = & \sup_{ x' \in B'} \left| x'\left(x\right) - \sum_{ i = 1}^n x'\left(f_i\right)g_i\left(x\right) \right| = \nonumber \\
& = & \sup_{ x' \in B'} \left| \sum_{ i = n + 1}^{\infty} x'\left(f_i\right)g_i\left(x\right)\right| = \sup_{\phi' \in C'} \left| \phi'\left(\sum_{ i = n + 1}^{\infty} g_i\left(x\right)e_i\right)\right| \nonumber \\
& = &  q\left( \sum_{ i = n+1}^{\infty} g_i\left(x\right)e_i \right)
\end{eqnarray*}
where $q$ is a continuous seminorm in $\bigwedge_{\beta}$. Then, $q\left( \sum_{ i = n+1}^{\infty} g_i\left(x\right)e_i \right)$ converges to 0 as $n$ converges to infinity due to the fact that $(g_i(x))_i = \sum_{i = 1}^{\infty} g_i(x) e_i$ in $\Lambda$.
\end{description}

To conclude, observe that, if (iii) and (iv) hold, then $((g_i)_i, (f_i)_i)$ and $((f_i)_i, (g_i)_i)$ are Schauder frames for $E$ and $E'$ respectively. By \cite[Proposition 2.3]{bfgr} we obtain that $((g_i)_i, (f_i)_i)$ is a shrinking Schauder frame.
\end{proof}

A locally convex algebra is a lcs which is an algebra with separately continuous multiplication. The spectrum of the algebra is the set of all non-zero multiplicative linear functionals. The following remark will be useful in Section \ref{sec:hormander}.
\begin{remark}\label{remark:algebras}\rm{\begin{itemize}
\item[(i)] In many cases $E$ is continuously included in a locally convex algebra $\mathcal{A}$ with non-empty spectrum, $\Lambda$ is a solid sequence space,  $(g_i)_i$ is a $\Lambda$-frame and every $g_i$ is the restriction to $E$ of a continuous linear multiplicative functional  on $\mathcal{A}.$ Let us assume that for some  $a\in \mathcal{A}$ the operator $$T:E \rightarrow E (\subset \mathcal{A}),\ \ x\mapsto  ax$$ is well defined and it is a topological isomorphism into, and that $\alpha:=(g_i(a))_i$ defines by pointwise multiplication a continuous operator on $\Lambda.$ Then, $(h_i)_i,$ where $$h_i:=\left\{\begin{array}{lcr} g_i, & \mbox{ if } & g_i(a)\neq 0 \\  0 & \mbox{ if } & g_i(a)=0 \end{array}\right.$$ is a $\Lambda$-frame. In fact, since $U\circ T$ is a topological isomorphism into then for every continuous seminorm $p$ on $E$ there is a continuous seminorm $q$ on $\Lambda$ such that
$$
p(x)\leq q\left((g_i(ax))_i\right) = q\left((g_i(a)g_i(x))_i\right) = q\left((g_i(a)h_i(x))_i\right).
$$ Finally, since the pointwise multiplication with $(g_i(a))_i$ is a continuous operator on $\Lambda$ we find a continuous seminorm $r$ on $\Lambda$ with
$$
p(x) \leq r\left((h_i(x))_i\right),\ \ x\in E.
$$

    \item[(ii)] If $E$ is a locally convex algebra with non-empty spectrum, $\Lambda$ is a barrelled  sequence space and $(g_i)_i$ is a $\Lambda$-frame consisting of  continuous linear multiplicative functionals  on $E,$ then $U(E)$ is a locally convex algebra under pointwise multiplication. Hence, if $E$ has no zero-divisors, the analysis map $U$ cannot be surjective. In fact, if there are $x,y\in E$ such that $U(x) = e_1$ and $U(y) = e_2$ then $$U(x\cdot y) = (g_i(xy))_i = (g_i(x)g_i(y))_i = 0$$ and the injectivity of $U$ implies $x\cdot y = 0,$ which is a contradiction. Since the range of $U$ is a topological subspace of $\Lambda,$ the non-surjectivity of $U$ implies the non-injectivity of the transposed map $U'.$ Consequently the expression of any element in $E'$ as a convergent series $$\sum_i \alpha_i g_i$$ with $\alpha \in \Lambda'$ is never unique.
        \end{itemize}}
        \end{remark}
\par\medskip
\section{$\Lambda$-frames in $\left(LB\right)$-spaces}\label{sec:LB}

Let $E={\rm ind}_n(E_n, \norm{\cdot}_n)$ and $\Lambda={\rm ind}_n(\Lambda_n, r_n)$ be complete $\left(LB\right)$-spaces and $(g_i)_i\subset E'$ a $\Lambda$-Bessel sequence. Let $U:E\to \Lambda$ be the continuous and linear map of Definition \ref{FR_Pre_Def_Fr} and, for each $n\in {\mathbb N},$ consider the seminormed space $(F_n, q_n)$ where
$$
F_n=\{x\in E: U(x)\in \Lambda_n\}
$$ and $q_n(x):=r_n(U(x)).$ Let us consider the topologies on $E$
$$
(E,\tau_1) = {\rm ind}_n(E_n, \norm{\cdot}_n),\ (E,\tau_2) = {\rm ind}_n(F_n,q_n).
$$ Finally, denote by $\tau_3$ the topology on $E$ given by the system of seminorms $x\mapsto p(U(x)),$ when $p$ runs in $cs(\Lambda).$ Then
$$
\tau_1 \geq \tau_2 \geq \tau_3,
$$ but observe that $\tau_2$ and $\tau_3$ need not be even Hausdorff. This notation will be kept through all this section.

\par\medskip
We observe that \textit{$(g_i)_i\subset E'$ is a $\Lambda$-frame if, and only if, the former three topologies coincide.}
\par\medskip
The coincidence $\tau_1=\tau_2$ is easily characterized under the mild additional assumption that the closed unit ball of $\Lambda_n$ is also closed in $\omega.$ This is the case for all (weighted) $\ell_p$ spaces, $1\leq p \leq \infty,$ but not for $c_0.$

Applications of the results in this section for weakly sufficient sets will be given in Section \ref{Examples}. We refer to \cite[8.5]{PCB} for the behavior of bounded sets and convergent sequences in $\left(LB\right)$-spaces.

\begin{proposition}\label{frames_LB} Assume that the closed unit ball of $\Lambda_n$ is closed in $\omega.$ Then, $\tau_1=\tau_2$ if and only if $(F_n,q_n)$ is a Banach space for each $n.$
\end{proposition}
\begin{proof} Assume that $\tau_1=\tau_2,$ which in particular implies that $\tau_2$ is Hausdorff.  Since $(F_n,q_n)$ is continuously injected in $(E, \tau_2),$ $q_n$ is a norm. Moreover, if $x\in E,$ $x\neq 0,$ there is $n$ such that $x\in F_n,$ hence $q_n(x)>0.$ We have $0<q_n(x)=r_n(U(x)),$ which implies $U(x)\neq 0.$ Thus $U$ is injective. Let $(x_j)_j$ be a Cauchy sequence in $(F_n, q_n).$ Then, it converges to a vector $x$ in the complete $\left(LB\right)$-space $E,$ and therefore its image under the analysis map $(U(x_j))_j$ is convergent to $U(x)$ in $\Lambda.$ Now, given $\varepsilon >0$ we can find $j_0$ such that $r_n(U(x_j)-U(x_k)) \leq \varepsilon$ whenever $j,\, k\geq j_0.$ That is, for $k\geq j_0,$
$$
U(x_k)\in U(x_j)+(\alpha \in \Lambda_n: \, r_n(\alpha)\leq \varepsilon),$$ and then

$$U(x)\in \overline{U(x_j)+(\alpha \in \Lambda_n: \, r_n(\alpha)\leq \varepsilon)}^{\ \omega}$$ for all $j\geq j_0.$ By hypothesis we get $U(x)\in \Lambda_n$ and $r_n(U(x-x_j))\leq \varepsilon$ for all $j\geq k_0.$ Hence $(F_n,q_n)$ is a Banach space.
\par
The converse holds since, by the open mapping theorem, two comparable $\left(LB\right)$-topologies must coincide.
\end{proof}

The next result depends on Grothendieck's factorization theorem  (see \cite[24.33]{Meise_Vogt_1997_Introduction}).
\begin{corollary}\label{cor:factorization} Assume that the closed unit ball of $\Lambda_n$ is closed in $\omega.$ Then $\tau_1=\tau_2$ if and only if for each $n$ there are $m$ and $C$ such that $F_n \subset E_m$ and $$||x||_m\leq C q_n(x)$$ for each $x\in F_n.$
\end{corollary}

 \begin{proposition}\label{frames_LB_Montel} If $E$ is Montel and $\tau_1=\tau_2$ then $( g_i )_i $ is a $\Lambda$-frame for $E.$
     \end{proposition}
    \begin{proof} As in Proposition \ref{frames_LB}, $U$ is injective and each $(F_n,q_n)$ is a normed space. By Baernstein's lemma (see \cite[8.3.55]{PCB}), as $E$ is a Montel space and $\Lambda$ is a complete $\left(LB\right)$-space, it suffices to show that for each bounded subset $B$ of $\Lambda$, the pre-image $U^{-1}(B)$ is bounded in $(E, \tau_2).$ As $\Lambda$ is regular, because it is complete, there is $n$ such that $B$ is contained and bounded in $\Lambda_n$, hence $U^{-1}(B)$ is contained and bounded in $F_n$, therefore bounded in $E.$\end{proof}

  Our next result is an abstract version of \cite[Theorems 2 and 3]{abanin}. Recall that a $\left(DFS\right)$-space is an $\left(LB\right)$-space $E={\rm ind}_n(E_n, \norm{\cdot}_n)$ such that for each $n$ there is $m>n$ such that the inclusion map $E_n \rightarrow E_m$ is compact.

    \begin{theorem}\label{frames_DFS} Let $E={\rm ind}_n(E_n, \norm{\cdot}_n)$ be a (DFS)-space and let $\Lambda={\rm ind}_n(\Lambda_n, r_n)$ be a complete $\left(LB\right)$-space. Assume that the closed unit ball of $\Lambda_n$ is closed in $\omega.$ If $(g_i)_i\subset E'$ is a $\Lambda$-Bessel sequence then the following conditions are equivalent:
    \begin{itemize}
     \item[(i)] $(g_i)_i$ is a $\Lambda$-frame,
     \item[(ii)] The map $U:E\to \Lambda,\ U(x) = ( g_i\left(x\right) )_i,$ is injective and for every $n\in {\mathbb N}$ there exists $m > n$ such that $F_n \subset E_m.$
    \end{itemize}
          \end{theorem}
    \begin{proof} If (i) is satisfied, $\tau_1 = \tau_2 = \tau_3.$ The injectivity of $U$ follows as in the proof of Proposition \ref{frames_LB} and the rest of (ii) follow by Corollary \ref{cor:factorization}.
\par
We prove that (ii) implies (i). Without loss of generality we can assume that $E_n \subset E_{n+1}$ with compact inclusion, $E_n \subset F_n$ and $q_n(x) \leq \norm{x}_n$ and $\norm{x}_{n+1}\leq \norm{x}_n$ for every $x\in E_n.$ It suffices to show that, under condition (ii), the inclusion $F_n \subset E_{m+2}$ is continuous. In fact, this implies the coincidence of the topologies $\tau_1 = \tau_2,$ hence the $\Lambda$-frame property by Proposition \ref{frames_LB_Montel}.
     \par
     Fix $n \in \mathbb{N}$ and define
     $$
     B = \left\{x\in F_n:\ q_n(x)\leq 1,\ \ \norm{x}_{m+1} > 1\right\}
     $$ and
     $$
     A = \left\{y = \frac{x}{\norm{x}_{m+1}}:\ x\in B\right\}.
     $$
     We can assume $B$ is an infinite set since otherwise the inclusion $F_n\subset E_{m+1}$ is continuous and we are done. Let $(p_j)_j$ denote a fundamental system of seminorms for the Fr\'echet space $\omega.$ We {\it claim} that there are $j_0\in {\mathbb N}$ and $C > 0$ such that $$\norm{x}_{m+1} \leq C p_{j_0}(U(x))$$ for every $x\in B.$
     Otherwise there is a sequence $(y_j)_j \subset A$ such that
     \begin{equation}\label{converge_omega}
     p_j(U(y_j)) \leq \frac{1}{j!}.
    \end{equation}
Assume that $(y_j)_j$ is bounded in $E_m$. Then it would be relatively compact in
$E_{m+1}$.
Therefore,
there is  a subsequence $(y_s)_s$ of $(y_j)_j$ that converges
to $y$ in $E_{m+1}$.
Hence $U(y_s) \rightarrow U(y)$ in $\Lambda$, hence in $\omega$. We can
apply (\ref{converge_omega}) to conclude
that $U(y)=0$, hence $y=0$, since $U$ is injective. This contradicts
$||y_s||_{m+1}=1$ for all $s$. 
Consequently, $(y_j)_j$ is unbounded in $E_m.$ Hence, for $j_1 = 1$ there exists $j_2 > j_1$ such that
     $$
     \frac{1}{6\cdot 2^2}\norm{y_{j_2}}_m > 3\norm{y_{j_1}}_m.
     $$ There is $\psi$ in the unit ball $B_{E_m'}$ of $E_m'$ such that
$$
\frac{1}{6\cdot 2^2}\left|\psi(y_{j_2})\right| > 3\norm{y_{j_1}}_m > 2\left|\psi(y_{j_1})\right|.
$$ Since $U:E_m\to \omega$ is a continuous and injective map then $U':\omega'\to E_m'$ has $\sigma(E_m',E_m)$-dense range and we can find $\varphi_2\in \omega'$ such that
$$
\max_{k=1,2}\left|(\varphi_2\circ U - \psi)(y_{j_k})\right|
$$ is so small that
$$
\frac{1}{6\cdot 2^2}\left|\varphi_2(U(y_{j_2}))\right| > 3\norm{y_{j_1}}_m > 2 \left|\varphi_2(U(y_{j_1}))\right|.
$$ By condition (\ref{converge_omega}) there is $j_2'$ such that
$$
\left|\varphi_2(U(y_{j}))\right| < \left|\varphi_2(U(y_{j_2}))\right|,\ \ j > j_2'.
$$ Proceeding by induction it is possible to obtain a sequence $(\varphi_\ell)_\ell \subset \omega'$ and an increasing sequence $(j_\ell)_\ell$ of indices such that $\varphi_\ell\circ U\in B_{E_m'}$ and
    $$
    \frac{1}{\ell(\ell+1)2^{\ell}}\left|\varphi_\ell(U(y_{j_\ell}))\right| > 3\sum_{k=1}^{\ell-1}\norm{y_{j_k}}_m > 2\sum_{k=1}^{\ell-1}\left|\varphi_\ell(U(y_{j_k}))\right|
    $$ while
    $$
    \left|\varphi_\ell(U(y_{j_k}))\right| < \left|\varphi_\ell(U(y_{j_\ell}))\right| \ \ \forall k > \ell.
    $$ We now consider
    $$
    y = \sum_{k=1}^\infty\frac{1}{k 2^k}y_{j_k} \in E_{m+1}.
    $$ Then
    $$
    \varphi_\ell(U(y)) = \sum_{k=1}^\infty\frac{1}{k 2^k}\varphi_\ell(U(y_{j_k})),
    $$ hence
    $$
    \begin{array}{ll}
    \left|\varphi_\ell(U(y))\right| & \begin{displaystyle} \geq \frac{1}{\ell 2^\ell}\left|\varphi_\ell (U(y_{j_\ell}))\right| - \sum_{k < \ell}\frac{1}{k 2^k}\left|\varphi_\ell(U(y_{j_k}))\right| -\sum_{k > \ell}^\infty\frac{1}{k 2^k}\left|\varphi_\ell(U(y_{j_k}))\right|\end{displaystyle}\\ & \\ & \begin{displaystyle}\geq \left(\frac{1}{\ell 2^\ell} - \sum_{k > \ell}\frac{1}{k 2^k}\right)\left|\varphi_\ell (U(y_{j_\ell}))\right| - \sum_{k < \ell}\frac{1}{k 2^k}\left|\varphi_\ell(U(y_{j_k}))\right|\end{displaystyle}\\ & \\ & \begin{displaystyle}\geq \frac{1}{\ell(\ell+1)2^\ell}\left|\varphi_\ell (U(y_{j_\ell}))\right| - \frac{3}{2}\sum_{k < \ell}\norm{y_{j_k}}_m\end{displaystyle}\\ & \\ & \begin{displaystyle}\geq \sum_{k < \ell}\norm{y_{j_k}}_m \geq \sum_{k < \ell}\norm{y_{j_k}}_{m+1} = \ell -1.\end{displaystyle}
    \end{array}
    $$ On the other hand $r_n(U(y_{j_k}))\leq 1$ for every $k\in {\mathbb N},$ which implies that the series
$$
\sum_{k=1}^\infty\frac{1}{k2^k}U(y_{j_k})
$$ converges in the Banach space $\Lambda_n.$ Hence $y\in F_n \subset E_m.$ Since $\varphi_\ell \circ U\in B'_{E_m}$ then $\left|\varphi_\ell(U(y))\right| \leq 1,$ which is a contradiction. Consequently the claim is proved and there are $j_0\in {\mathbb N}$ and $C > 0$ such that $$\norm{x}_{m+1} \leq C p_{j_0}(U(x))$$ for every $x\in B.$ In order to conclude that the inclusion $F_n\subset E_{m+2}$ is continuous, it suffices to check that $B$ is bounded in $E_{m+2}.$ To this end we first observe that
$$
1\leq \norm{x}_{m+1}\leq C p_{j_0}(U(x))\leq C'\norm{x}_{m+2}
$$ for some $C'>0$ and for all $x\in B.$ Then
$$
\left\{\frac{x}{\norm{x}_{m+2}}:\ x\in B\right\}\subset E_m
$$ is a bounded set in $E_{m+1},$ hence relatively compact in $E_{m+2}.$ We now proceed by contradiction and assume that $B$ is unbounded in $E_{m+2}.$ Then there exists a sequence $(x_j)_j\subset B$ with $\norm{x_j}_{m+2}\geq j.$ Passing to a subsequence if necessary we can assume that
$$
z_j:=\frac{x_j}{\norm{x_j}_{m+2}}
$$ converges to some element $z\in E_{m+2}$ such that $\norm{z}_{m+2} = 1.$ Since the inclusion $E_{m+2}\subset F_{m+2}$ is continuous we get
$$
\lim_{j\to \infty}q_{m+2}(z_j-z) = 0.
$$ From the injectivity of $U$ we get $q_{m+2}(z) = r_{m+2}(U(z)) = a > 0,$ and there is $j_0\in {\mathbb N}$ such that $q_{m+2}(z_j) \geq \frac{a}{2}$ whenever $j\geq j_0,$ which implies $q_{m+2}(x_j) \geq \frac{a}{2}j$ for all $j\geq j_0.$ This is a contradiction, since ($m > n$)
$$
q_{m+2}(x_j) \leq q_n(x_j) \leq 1.
$$ The proof is complete.

    \end{proof}

    \section{Examples}\label{Examples}

   \subsection{Weighted spaces of holomorphic functions}

Let $G$ be  either an open disc centered at the origin or $\mathbb{C}.$ A {\it radial weight} on $G$ is a  strictly positive continuous
function $v$ on $G$ such that $v(z)=v(|z|),\ z \in G.$ Then,  the {\it weighted Banach space of holomorphic functions}
is defined by
\begin{center}
$Hv(G) := \left\{ f \in {\mathcal H}(G) \ : \  ||f||_v := \sup_{z \in G} v(|z|) |f(z)| <
+ \infty \right\}.$
\end{center}
\par\medskip
Let $V=(v_n)_n$ be a decreasing sequence of weights on $G.$ Then the weighted inductive limit of spaces of holomorphic functions is defined by $$VH:={\rm ind_n}\ Hv_n(G),$$ that is, $VH(G)$ is the increasing union of the Banach spaces $Hv_n(G)$ endowed with the strongest locally convex topology for which all the injections $Hv_n(G) \to VH(G)$ become continuous.

Similarly, given an increasing sequence of weights $W=(w_n)_n$ on $G,$  the weighted projective limit of spaces of entire functions  is defined by $$HW(G):={\rm proj_n}Hw_n(G),$$
that is, $HW(G)$ is the decreasing intersection of the Banach spaces $Hw_n(G)$ whose topology is defined by the sequence of norms $|| \cdot||_{w_n}.$ It is a Fr\'echet space.

In both cases, when $G=\mathbb{C}$ we will simply write $VH$ and $HW.$

 Given any sequence $S:=(z_i)_i$ and a decreasing sequence of weights $V$ on $G$, put
$$\nu_n(i)=v_n(z_i)$$
and $$V\ell_\infty (S)={\rm ind}_n \ell_\infty(\nu_n).$$ For an increasing sequence of weights $W=(w_n)_n$ on $G,$ put $$\omega_n(i):=w_n(z_i)$$ and $$\ell_\infty W(S)=\bigcap_n\ell_\infty(\omega_n).$$ Obviously, the restriction maps $$R:VH(G) \to V\ell_\infty (S), \, f\mapsto (f(z_i))_i$$ and $$R:HW(G) \to \ell_\infty W (S), \, f\mapsto (f(z_i))_i$$
are well defined and continuous, that is, $(\delta_{z_i})_i$ is a $V\ell_\infty (S)$-Bessel sequence for $VH(G)$ and a $\ell_\infty W (S)$-Bessel sequence for $HW(G).$ We want to analyze when these Bessel sequences are in fact frames, that is, when the restriction map is an isomorphism into.

Let us first concentrate on the Fr\'echet case.
Then,  $(\delta_{z_i})_i$ is a  $\ell_\infty W (S)$-frame if and only if for every $n$ there are $m$ and $C$ such that $$ \sup_{z\in G}|f(z)|w_n(z)\leq C \sup_{i}|f(z_i)|w_m(z_i)$$ for every $f\in HW(G).$
This is the same as saying that $S$ is a {\it sufficient set} for $HW(G).$ The concept of sufficient set was introduced by Ehrenpreis in \cite{Ehrenpreis}.
\par\medskip
The $\left(LB\right)$-case is more delicate. Following the notation of section \ref{sec:LB}, if $E_n:=Hv_n(G),$ the space $F_n:=\left\{f\in VH(G): R(f)\in \ell_\infty(\nu_n)\right\}$ is usually denoted by $A(S,v_n)$ and the corresponding seminorm $q_n$ is denoted $\norm{\cdot}_{n,S},$ that is,
$$
\norm{f}_{n,S} = \sup_{i\in {\mathbb N}}\left|f(z_i)\right|\nu_n(i),\ \ f\in A(S,v_n).
$$
Then $\tau_1$ is the topology of the inductive limit $VH(G)$ and $\tau_2$ is the one of ${\rm ind}_n A(S,v_n).$  We recall that $S$ is said to be {\it weakly sufficient} for $VH(G)$ when $VH(G)={\rm ind}_n A(S,v_n)$ topologically, i.e.\ $\tau_1 = \tau_2$. It should be mentioned that this definition a priori is not restricted to discrete sets, but this is the most interesting case.  We obtain the following general results, that have been formulated in several papers in one way or other in concrete situations. See in particular \cite{korobeinik1}, \cite{korobeinik2}, \cite{napalkov1} and \cite{napalkov2}.

\begin{theorem} The following statements are equivalent:
\begin{itemize}

 \item[(i)] $S:=(z_i)_i$ is weakly sufficient.

 \item[(ii)] $A(v_n,S)$ is a Banach space for every $n\in\mathbb{N}.$

 \item[(iii)] For each $n$ there are $m\geq n$ and  $C>0$ such that for every $f\in VH(G)$ one has $$||f||_m\leq C ||f||_{n,S}.$$
     \end{itemize}
     \end{theorem}

     \noindent
     \begin{proof} Apply Proposition \ref{frames_LB} and \ref{cor:factorization}.
     \end{proof}

\begin{theorem}\label{effective set} Let us assume that $\frac{v_{n+1}}{v_n}$ vanishes at infinity on $G$ for every $n\in {\mathbb N}.$ Then, the following conditions are equivalent:
\begin{itemize}
 \item[(i)] $S:=(z_i)_i$ is weakly sufficient.
\item[(ii)] The restriction map $VH(G) \to V\ell_\infty (S)$ is injective and for each $n$ there are $m\geq n$ and  $C>0$ such that $A(v_n,S) \subset Hv_m(G).$
\end{itemize}
\end{theorem}

\noindent
\begin{proof} It follows from Proposition \ref{frames_DFS}.\end{proof}
\par\medskip\noindent
The injectivity of the restriction map means that $S$ is a {\it uniqueness set} for $VH(G).$ As a consequence of Proposition \ref{frames_LB_Montel} we obtain

\begin{theorem}\label{ws-Montel}
If $VH(G)$ is Montel, $S$ is  weakly sufficient if and only if
 the restriction map $$R:VH(G) \to V\ell_\infty(S), \, \, f \, \mapsto f|_S,$$ is a topological isomorphism into.
  \end{theorem}

Theorem \ref{ws-Montel} asserts that, if the space $VH(G)$ is Montel, the set of  Dirac evaluations $\{ \delta_z \ | \ z \in S\} \subset VH(G)'$ is a $V\ell_\infty(S)$-frame if and only if $S$ is  weakly sufficient.

 If the sequence $V=(v)$ reduces to one weight, $(\delta_{z_i})_i$ is a $\ell_\infty (\nu)$-frame for $Hv(G)$ if and only if $S$ is a {\it sampling set} for $Hv(G).$ If $(v_n)_n$  is a decreasing  sequence of weights on $G$ and $S$ is a sampling set for $Hv_n(G)$ for each $n$, then $S$ is a weakly sufficient set for $VH(G).$  However, Khoi and Thomas \cite{khoi-thomas} gave examples of countable weakly sufficient sets $S=(z_i)_i$ in the space $$A^{-\infty}(\mathbb{D}):={\ind}_nHv_n(\mathbb{D}), \mbox{ with }\, v_n(z)=(1-|z|)^n,$$ which are not sampling sets for any $Hv_n(\mathbb{D}),$ $n\in \mathbb{N}.$ As $A^{-\infty}(\mathbb{D})$ is Montel, $(\delta_{z_i})_i$ is a $V\ell_\infty(S)$-frame for $A^{-\infty}(\mathbb{D})$ which is not a $\ell_\infty (\nu_n)$-frame for $Hv_n(\mathbb{D})$ for any $n.$ Bonet and Domanski \cite{Bonet_Domanski_2003_Sampling} studied weakly sufficient sets in $A^{-\infty}(\D)$ and their relation to what they called $(p,q)$-sampling sets.
\par\medskip
The dual of the space $A^{-\infty}(\mathbb{D})$ can be identified via the Laplace transform with the space of entire functions $A^{-\infty}_{\mathbb D}:= HW(\mathbb{C})$ for the sequence of weights $W=(w_n)_n$, $$w_n(z)=(1+|z|)^ne^{-|z|},$$ (see \cite{melikhov} and also \cite{abanin_khoi} for the several variables case). In \cite{abanin_khoi} explicit constructions of sufficient sets for this space are given. For instance, for each $k$ take  $\ell_k \in \mathbb{N}, \, \ell_k>2\pi k^2,$ and let $z_{k,j}:=kr_{k,j}$, $1\leq j \leq \ell_k,$ where $r_{k,j}$ are the $\ell_k$-roots of the unity, then, with an appropriate order, $(\delta_{z_{k,j}}: k\in {\mathbb N}, 1\leq j \leq \ell_k)$ is a $\ell_\infty W(S)$-frame in $A^{-\infty}_{\mathbb D}.$  More examples for non-radial weights can be found in \cite{Abanin_Varziev}.

Finally, from Proposition \ref{frames-representing},  we recover the following consequence about representing systems. It should be compared with Corollary \ref{cor:rep-Ap} below.
\begin{theorem}\label{abanin_khoi}{\rm (\cite{abanin_khoi})}
\begin{itemize}
 \item[(i)] $(\lambda_k )_k \subset \mathbb{C}$ is   sufficient for $A^{-\infty}_{\mathbb D}$ if and only if every function $f\in A^{-\infty}(\D)$ can be represented as $$f(z)=\sum_{k}\alpha_ke^{\lambda_kz}$$ where
$$
\sum_{k}\left|\alpha_k\right|(1+|\lambda_k|)^{-n}e^{|\lambda_k|} < \infty\ \mbox{for some}\ n\in {\mathbb N}.
$$
        \item[(ii)]  $(\lambda_k)_k \subset \mathbb{D}$ is weakly sufficient in $A^{-\infty}(\D)$ if and only if each function $f\in A^{-\infty}_{\mathbb D}$ can be represented as $$f(z)=\sum_{k}\alpha_ke^{\lambda_kz}$$ where
$$
\sum_k\left|\alpha_k\right|(1-|\lambda_k|)^{-n} < \infty\ \mbox{for every}\ n\in {\mathbb N}.
$$
     \end{itemize}

     \end{theorem}

\subsection{The H\"ormander algebras }\label{sec:hormander}

In this Section we use  Landau's notation of little $o$-growth and capital $O$-growth. A function $p: \mathbb{C} \rightarrow [0, \infty [$ is called a \textit{growth condition} if it is continuous, subharmonic, radial, increases with $|z|$ and satisfies

\begin{itemize}
\item [$(\alpha)$] $\log(1+|z|^2)=o(p(|z|))$ as $|z|\to \infty$,
\item [$(\beta)$] $p(2|z|)=O(p(|z|))$ as $|z|\to \infty.$
\end{itemize}

Given a growth condition $p$, consider the weight $v(z)=e^{-p(|z|)},$ $z\in \mathbb{C},$ and the decreasing sequence of weights $V=(v_n)_{n},$ $v_n=v^n.$ We define the following weighted spaces of entire functions (see e.g.\ \cite{Berenstein_Taylor_1979_a}, \cite{Berenstein_Gay_1995_complex}):
$$A_p:=\left\{f\in {\mathcal H}(\mathbb{C}): \ {\rm there \ is} \ A>0: \sup_{z\in \mathbb{C}}|f(z)| \exp(-Ap(z))<\infty \right\},$$
that is, $A_p=VH,$
endowed with the inductive limit topology, for which it is a $\left(DFN\right)$-algebra (see e.g. \cite{Meise_1985_sequence}). Given any sequence $S = (z_i)_i$ we will denote $A_p(S) = V\ell_\infty(S),$ that is,
$$
A_p(S) = \bigcup_n\ell_\infty(\nu_n),\ \ \nu_n(i) = e^{-np(|z_i|)}.
$$

If we consider the increasing sequence of weights $W=(w_n)_{n},$ $w_n=v^{1/n},$  we define

$$A^0_p:=\left\{f\in {\mathcal H}(\mathbb{C}): \ {\rm for \ all} \ \varepsilon>0: \sup_{z\in \mathbb{C}}|f(z)| \exp(-\varepsilon p(z))<\infty \right\},$$
that is, $A_p^0=HW,$
endowed with the projective limit topology, for which it is a nuclear Fr\'echet algebra (see e.g. \cite{Meise_Taylor_1987_sequence}). Clearly $A^0_p \subset A_p$. As before, given a sequence $S = (z_i)_i$ we will denote  $A_p^0(S) = \ell_\infty W(S),$ that is,
$$
A_p^0(S) = \bigcap_n\ell_\infty(\omega_n),\ \ \omega_n(i) = e^{-\frac{1}{n}p(|z_i|)}.
$$
\par\medskip
Condition $(\alpha)$ implies that, for each $a>0,$ the weight $v_a(z):=e^{-ap(|z|)}$ is rapidly decreasing, consequently, the polynomials are contained and dense in $H^0_{v_a},$ and that for $a<b$ the inclusion $H_{v_a} \subset H^0_{v_b}$ is compact. Therefore the polynomials are dense in $A_p$ and in $A^0_p.$
 Condition $(\beta)$ implies that both spaces are stable under differentiation. By the closed graph theorem, the differentiation operator $D$ is continuous on $A_p$ and on $A^0_p.$

Weighted algebras of entire functions of this type, usually known as H\"ormander algebras, have been considered since the work of Berenstein and Taylor \cite{Berenstein_Taylor_1979_a} by many authors; see e.g.\ \cite{Berenstein_Gay_1995_complex} and the references therein.

As an example, when $p_a(z)=|z|^a$, then $A_{p_a}$ consists of all entire functions of order $a$ and finite type or order less than $a,$ and $A^0_{p_a}$ is the space of all entire functions of order at most $a$ and type $0$. For $a=1$, $A_{p_1}$ is the space of all entire functions of exponential type, also denoted $Exp(\mathbb{C})$ and $A^0_{p_1}$ is the space of  entire functions of infraexponential type.

As it is well-known, the Fourier-Borel transform $\mathcal{F}:H(\C)'\rightarrow Exp(\C)$ defined by $\mathcal{F}(\mu) := \widehat{\mu},$ where $\widehat{\mu}(z) := \mu_{\omega}(e^{z \omega}),$ is a topological isomorphism. As a consequence, the dual space of $Exp(\C)$ can be identified with the space of entire functions, $H(\C).$ In the same way, for $a>1$ and $b$ its conjugate exponent ($a^{-1}+b^{-1}=1$) via the Fourier-Borel transform $\mathcal{F} $ we have the following identifications \cite{Taylor} $$(A_{p_a})'=A^0_{p_b},\ \mbox{and}\ (A^0_{p_a})'=A_{p_b}.$$

From every  (weakly) sufficient set $(z_j)_{j=1}^\infty$ for ($A_p$) $A^0_p$  we can remove finitely many points $(z_j)_{j=1}^N$ and still we have a (weakly) sufficient set (see \cite[Corollary to Proposition 4]{Abanin_Varziev}). In fact, take $Q$ a non constant polynomial which vanishes precisely at points $(z_j)_{j=1}^N.$ Since the multiplication operator $$T_Q(f)(z)=Q(z)f(z)$$ is a topological isomorphism from $A_p$ (resp. $A_p^0$) into itself and pointwisse multiplication by $(Q(z_j))_j$ is continuous on $A_p(S)$ (resp. $A_p^0(S)$) it suffices to apply \ref{remark:algebras}(i).
\par\medskip

Now, we give examples of frames of type $(\delta_{z_i})_i$ in these algebras. We deal first with the Fr\'echet case.

\begin{theorem} Given a growth condition $q$ let $S:=(z_n)_n$ be a sequence in $\mathbb{C}$ with $\lim_j |z_j|=\infty $  and assume that there is $C>0$ such that the distance  $d(z,S)$ satisfies $d(z,S)\leq C|z|/\sqrt{q(|z|)}$ for all $z\in \mathbb{C}.$ Then, the sequence  $(\delta_{z_j})_j$ is a $A_p^0(S)$-frame for $A_p^0$ whenever $p(r)=o(q(r))$ as $r\to \infty.$
\end{theorem}
\begin{proof}
 We take $V(r)=q(r).$ The family $\{ap, \, a>0\}$ satisfies (i), (ii) and (iii) in \cite[p.178]{Schneider} and the conclusion follows after applying \cite[Theorem 5.1]{Schneider}.
\end{proof}

\par\medskip\noindent
In particular, if $p(r)=o(r^2)$ as $r\to \infty$ we may take $q(r)=r^2.$
\begin{corollary}\label{cor:Ap} If $p(r)=o(r^2),$ then for arbitrary $\alpha, \, \beta >0$ the regular lattice $\{\alpha m+i \beta m: \, n, \, m\in \mathbb{Z}\}$ is a sufficient set for $A_p^0(\mathbb{C}).$ In other words, if $S = (z_{n,m})$ where $z_{n,m}:=\alpha n+i \beta m$ then the sequence $(\delta_{z_{n,m}})$ is a $A_p^0(S)$-frame for $A_p^0(\mathbb{C}).$
\end{corollary}

\par\medskip\noindent
The former result is also true in the limit case $p(r)=r^2.$ In fact,

\begin{proposition}\label{prop:Ap} If $p(r)=r^2,$ then for arbitrary $\alpha, \, \beta >0$ the regular lattice $\{\alpha m+i \beta m: \, n, \, m\in \mathbb{Z}\}$ is a sufficient set for $A_p^0(\mathbb{C}).$ In other words, if $S = (z_{n,m})$ where $z_{n,m}:=\alpha n+i \beta m$ then the sequence $(\delta_{z_{n,m}})$ is a $A_p^0(S)$-frame for $A_p^0(\mathbb{C}).$
\end{proposition}

\begin{proof} First, we observe that in this case, $A_p^0(\mathbb{C})$ coincides algebraically and topologically with the intersection $$\bigcap _{\gamma>0}{\mathcal F}^2_\gamma$$ of the Bargmann-Fock spaces $${\mathcal F}^2_\gamma:=\left\{ h\in H(\mathbb{C}): ||f||_\gamma := \int_\mathbb{C}|f(z)|^2e^{-\gamma |z|^2}dz <\infty\right\}.$$ Then, by \cite{seip_walsten}, there is $\gamma_0$ such that for $\gamma \geq \gamma_0$ we find constants $A_\gamma, \, B_\gamma$ such that $$A_\gamma ||f||^2_\gamma \leq \sum _{n,m}|f(z_{n,m})|^2e^{-\gamma |z_{n,m}|^2}\leq B_\gamma ||f||^2_\gamma.$$
To finish, it is enough to observe that in the definition of $A_p^0(S)$ one can replace the $\ell_\infty$ norms by $\ell_2$-norms.
\end{proof}

According to \cite{Schneider} (see the comments after Corollary 4.9) there is an entire function of order $2$ and finite type which vanishes at the lattice points $S = \{n+im:\,n, m\in \mathbb{Z}\}.$ In the case $r^2=o(p(r))$ we have $f\in A^0_p,$ and the restriction map defined on $A_p^0$ by $f\mapsto f|_S$ is not injective. Consequently, the lattice points are not a sufficient set for $A^0_p.$ Similarly, the lattice points are not a weakly sufficient set for $A_p$ in the case $r^2= O(p(r)).$

\par\medskip

From \cite[Proposition 8.1]{Schneider} and Theorem \ref{effective set} we get

\begin{proposition} If $p(r)=o(r^2),$ then for arbitrary $\alpha, \, \beta >0$ the regular lattice $\{\alpha m+i \beta m: \, n, \, m\in \mathbb{Z}\}$ is a weakly sufficient set for $A_p(\mathbb{C}).$ In other words, if $S = (z_{n,m})$ where $z_{n,m}:=\alpha n+i \beta m$ then the sequence $(\delta_{z_{n,m}})$ is a $A_p(S)$-frame for $A_p(\mathbb{C}).$
\end{proposition}

In particular, for the space $Exp(\mathbb{C}),$ the sequence $(\delta_{n+im})_{n,m\in \mathbb{Z}}$ is a $A_p(S)$-frame \cite[Theorem 1]{Taylor_1972_Discrete}. Here $p(z) = |z|$ and $S = (n+im)_{n,m\in \mathbb{Z}}.$
\par\medskip

By Proposition \ref{frames-representing} if $S=(z_i)_i\subset G$ is a discrete (weakly) sufficient set in $HW(G)$ (resp. in $VH(G)$) each element in the dual space can be represented as a convergent series of type $$\sum_i \alpha_i \delta_{z_i}$$ with coefficients in a given sequence space. Since the spaces under consideration are algebras, this representation is not unique by \ref{remark:algebras}(ii). As in many cases the dual space can be identified with a weighted space of holomorphic functions (via the Laplace or the Fourier-Borel transform) in such a way that point evaluations $\delta_{z_i}$ are identified with the exponentials $e^{z_iz}$, therefore we get a representation of the elements in the dual space as Dirichlet series, thus obtaining as a consequence several known results, for instance:

\begin{corollary}{\rm (\cite{Taylor_1972_Discrete})}
Every entire function $f(z)$ can be represented in the form $$f(z) = \sum_{n,m = -\infty}^{\infty} a_{n,m}e^{(n + im)z}$$ where $|a_{n,m}|e^{k(n^2 + m^2)^{1/2}} \to 0$ as $n^2 + m^2 \to +\infty$ for every $k > 0.$ Such expansion of $f$ is never unique.
\end{corollary}

\begin{corollary}\label{cor:rep-Ap} For $a\geq 2$ every function $f\in A_{p_a}$ and can be represented in the form $$f(z) = \sum_{n,m = -\infty}^{\infty} a_{n,m}e^{(n + im)z}$$ with coefficients $(a_{n,m})$  satisfying $$|a_{n,m}|\leq C{\rm exp}(-\varepsilon(n^2+m^2)^{b/2})$$ ($b$ the conjugate of $a$) for some constants $\varepsilon,\, C>0.$
 \end{corollary}
\begin{proof}
 According to Corollary \ref{cor:Ap} and Proposition \ref{prop:Ap} the sequence $S = \{e^{(n+im)z}:\ n,m\in {\mathbb Z}\}\subset A_{p_a}$ is a $A_{p_b}^0(S)$-frame for $A_{p_b}^0.$ Since the dual space of $\Lambda = A_{p_b}^0(S)$ is
$$
\Lambda' = \left\{(a_{n,m}):\ |a_{n,m}|{\rm exp}(\varepsilon(n^2+m^2)^{b/2})< \infty\ \mbox{for some}\ \varepsilon > 0\right\}
$$ it suffices to apply Proposition \ref{frames-representing} to conclude.
\end{proof}

\par\medskip

\textbf{Acknowledgements.}
The present research  was partially supported by the projects  MTM2013-43450-P and GVA Prometeo II/2013/013 (Spain).

\par

\noindent \textbf{Authors' addresses:}\\

Jos\'e Bonet: Instituto Universitario de Matem\'{a}tica Pura y Aplicada IUMPA,
Universitat Polit\`{e}cnica de Val\`{e}ncia,  E-46071 Valencia, Spain

email: jbonet@mat.upv.es \\

Carmen Fern\'andez and Antonio Galbis:
Departamento de An\'alisis Matem\'atico,
Universitat de Val\`encia,
E-46100 Burjasot (Valencia), Spain.

email: fernand@uv.es and antonio.galbis@uv.es \\

Juan M. Ribera:
Instituto Universitario de Matem\'{a}tica Pura y Aplicada IUMPA,
Universitat Polit\`{e}cnica de Val\`{e}ncia,  E-46071 Valencia, Spain.

email: juaripuc@mat.upv.es

\end{document}